\documentclass{amsart}

\usepackage{amsfonts}
\usepackage{amssymb}
\usepackage{amsmath}
\usepackage{amsthm}
\usepackage{mathrsfs}
\usepackage{comment}
\usepackage{caption}
\usepackage{subcaption}
\usepackage[curve]{xypic}
\usepackage{manfnt}
\usepackage{color}

\newtheorem{theorem}{Theorem}[section]
\newtheorem{proposition}[theorem]{Proposition}

\newtheorem{corollary}[theorem]{Corollary}
\newtheorem*{thmA}{Theorem A}
\newtheorem*{thmB}{Theorem B}
\newtheorem*{thmC}{Theorem C}

\theoremstyle{definition}

\theoremstyle{remark}
\newtheorem{rmk}[theorem]{Remark}
\newtheorem{example}[theorem]{Example}

\newcommand{\scrD}{\mathscr{D}}
\newcommand{\scrU}{\mathscr{U}}
\newcommand{\calF}{\mathcal{F}}
\newcommand{\calG}{\mathcal{G}}
\newcommand{\calL}{\mathcal{L}}
\newcommand{\bfG}{\mathbf{G}}
\newcommand{\bfB}{\mathbf{B}}
\newcommand{\bfT}{\mathbf{T}}
\newcommand{\bfX}{\mathbf{X}}
\newcommand{\bfY}{\mathbf{Y}}
\def\bbF{{\mathbb{F}}}
\def\bbZ{{\mathbb{Z}}}
\def\bbK{{\mathbb{K}}}
\def\bbQ{{\mathbb{Q}}}
\def\bfW{{\mathbf W}}
\def\bfb{{\mathbf b}}
\def\bfc{{\mathbf c}}
\def\bfx{{\mathbf x}}
\def\bfw{{\mathbf w}}
\def\bfO{{\mathbf O}}
\def\D{\Delta}
\def\bpi{{\boldsymbol{\pi}}}

\newcommand{\simto}{\,\mathop{\rightarrow}\limits^\sim\,}

\newcommand{\longiso}{\stackrel{\sim}{\longrightarrow}}
\newcommand{\isolong}{\stackrel{\sim}{\longleftarrow}}

\DeclareMathOperator{\Ind}{{\mathrm{Ind}}}
\DeclareMathOperator{\Res}{{\mathrm{Res}}}
\DeclareMathOperator{\Hom}{{\mathrm{Hom}}}
\DeclareMathOperator{\For}{{\mathrm{For}}}
\DeclareMathOperator{\Indu}{{\underline{\mathrm{Ind}}}}

\def\lexp#1#2{\kern\scriptspace\vphantom{#2}^{#1}\kern-\scriptspace#2}
\newcommand{\flag}{\mathcal{B}}
\newenvironment{dedication}
        {\vspace{0ex}\begin{quotation}\begin{center}\begin{em}}
        {\par\end{em}\end{center}\end{quotation}}

\title[Translation by the full twist and Deligne--Lusztig varieties]{Translation by the full twist \\ and Deligne--Lusztig varieties}

\author{C\'edric Bonnaf\'e}
\address{C. Bonnaf\'e, IMAG, Universit\'e de Montpellier, CNRS, Montpellier, France} 
\makeatletter
\email{cedric.bonnafe@univ-montp2.fr}
\makeatother

\author{Olivier Dudas}
\address{O. Dudas, IMJ-PRG, CNRS and Universit\'e Paris Diderot, 
UFR de Math\'ematiques, 
B\^atiment Sophie Germain,
5 rue Thomas Mann,
75205 Paris CEDEX 13,
France}
\email{olivier.dudas@imj-prg.fr}

\makeatother
\author{Rapha\"el Rouquier}
\address{R. Rouquier, UCLA Mathematics Department
Los Angeles, CA 90095-1555, 
USA}
\email{rouquier@math.ucla.edu}

\makeatother

\date{\today}

\thanks{The first and second authors gratefully acknowledge financial support by the ANR, Project
No ANR-16-CE40-0010-01. The third author gratefully acknowledges financial support 
by the NSF (grant DMS-1161999 and DMS-1702305) and by a
grant from the Simons Foundation (\#376202)}

\begin{document}

\begin{abstract}
We prove several conjectures about the cohomology of Deligne--Lusztig varieties: invariance under conjugation in the braid group, behaviour with respect to translation by the full twist, parity vanishing of the cohomology for the variety associated with the full twist. In the case of split groups of type $A$, and using previous results of the second author, this implies Brou\'e--Michel's conjecture on the disjointness of the cohomology for the variety associated to any {\it good} regular element. That conjecture was inspired
by Brou\'e's abelian defect group conjecture and the specific form Brou\'e
	conjectured for finite groups of Lie type \cite[R\^eves 1 et 2]{Br}.
\end{abstract}

\maketitle

\begin{dedication}
To Michel Brou\'e, for sharing his dreams.
\end{dedication}

\bigskip

%%%%%%%%%%%%%%%%%%%%%%%%%%%%%%%%%%
\section*{Introduction}
%%%%%%%%%%%%%%%%%%%%%%%%%%%%%%%%%%
Let $\bfG$ be a connected reductive algebraic group over $\overline{\bbF}_p$ and $F$
be an 
endomorphism of $\bfG$, a power of which is a Frobenius endomorphism. The finite group $\bfG^F$ is a finite reductive group. In 1976, Deligne and Lusztig  \cite{DeLu} defined 
quasi-projective varieties $\bfX(wF)$ attached to elements $w$ of the Weyl group $W$. These \emph{Deligne--Lusztig varieties} are endowed with an action of $\bfG^F$ and therefore their $\ell$-adic cohomology groups $H_c^i(\bfX(wF))$ are finite-dimensional representations of $\bfG^F$. 

\smallskip

The explicit decomposition of the virtual characters $\sum (-1)^i H_c^i(\bfX(wF))$ was
given by Lusztig (see \cite[Thm. 4.23]{LuBook}) and played a key role
in his classification of the unipotent characters of $\bfG^F$. Much less is known about 
individual cohomology groups, but Brou\'e \cite{Br,Brmit}, Brou\'e-Malle \cite{BrMa} and Brou\'e-Michel \cite{BM}
formulated conjectures about
those individual cohomology groups, for particular elements $w$ (good regular
elements).

\smallskip

Let $\flag \times \flag$ be the double flag variety of $\bfG$. The diagonal
$\bfG$-orbits are parametrized by the Weyl group, and the Deligne--Lusztig variety $\bfX(wF)$ is the intersection of the orbit $\bfO(w)$ attached to $w$ with the graph of $F$. Therefore the cohomology with compact support of $\bfX(wF)$ is the cohomology of
the restriction of the constant sheaf on $\bfO(w)$ to the graph of $F$ . This construction works more generally with any $\bfG^F$-equivariant constructible complex on the
double flag variety and this provides a triangulated functor
  $$\Indu_F : D_{\bfG^F}^b(\flag \times \flag) \longrightarrow D^b(\bfG^F \text{-mod})$$
whose image is the category of unipotent representations. This functor and its right adjoint were recently studied by Lusztig in \cite{Lu15}, in relation with character
sheaves.

\smallskip

There are several advantages of working with general sheaves. First, one can construct standard objects in $D_{\bfG^F}^b(\flag \times \flag)$ attached to any element of the braid group of $W$, not just elements of $W$. Second, the category $D_{\bfG^F}^b(\flag \times \flag)$ has a convolution for which the standard objects are invertible. Using further properties of the induction functor one shows that the cohomology of a Deligne--Lusztig variety $\bfX(wF)$ depends only on the conjugacy class of $wF$ in the braid group. This was conjectured in \cite[Conj.~3.1.7]{DMR06} and originally proved by Deligne--Lusztig in the case of conjugation by a sequence of cyclic shifts (see \cite[Thm. 1.6, case 1.]{DeLu}).

\begin{thmA}
	Let $w,w' \in W$. Assume that $w F$ and $w'F$ are  conjugate under the braid group of $W$. Then for every $i \in \bbZ$ the $\bfG^F$-modules $ H_c^i\big(\bfX(wF)\big)$ and  $H_c^i\big(\bfX(w'F)\big)$ are isomorphic. 
\end{thmA}

The categories $D_{\bfG^F}^b(\flag \times \flag)$ and  $D^b(\bfG^F \text{-mod})$ are filtered by two-sided cells and families and Lusztig showed in \cite{Lu15} that the induction functor $\Indu_F$ and its adjoint respect these filtrations. The action by convolution of the standard object attached to the full twist $\bpi = (w_0)^2$ is a shift
on the subquotients of this filtration,
as shown by Bezrukavnikov--Finkelberg--Ostrik \cite[Remark 4.3]{charDmod}. We use this property to prove the following theorem which was conjectured by Digne--Michel--Rouquier \cite[Conj. 3.3.24]{DMR06}. Here $N$ is the number of positive roots of $W$ and $A_\rho$ is the degree of the generic degree of $\rho$. 

\begin{thmB}
	Let $\rho$ be a unipotent character of $\bfG^F$ and let $w \in W$. We have
	  $$ \big\langle \rho, H_c^i\big(\bfX(\boldsymbol\pi w F)\big) \big\rangle_{\bfG^F} = \big\langle\rho, H_c^{i-4N+2A_{\rho}}\big(\bfX(w F)\big)\big\rangle_{\bfG^F}$$
	for all $i \in \bbZ$.
\end{thmB}

In particular, the cohomology of the Deligne--Lusztig variety $\bfX(\boldsymbol\pi F)$ involves only principal series characters, and in even degrees only. The consequence of such result for $\mathrm{GL}_n(q)$ is that one can determine explicitly the cohomology of any Deligne--Lusztig variety attached to roots of $\bpi$ in the braid group,
using results of the second author \cite{Du}. In particular, the cohomology groups of different degrees have no irreducible constituent in common. This was conjectured by Brou\'e--Michel (see \cite[Conj. 5.7(1)]{BM}).

\begin{thmC}
	Assume that $(\bfG,F)$ is a (split) group of type $A$. Let $d \geq 1$ and $w \in W$ be a $d$-th root of $\bpi$ in the braid group. For all $i \neq j$, we have
	  $$ \big\langle H_c^i(\bfX(wF)) , H_c^j(\bfX(wF)) \big\rangle_{\bfG^F} = 0.$$
\end{thmC}

\bigskip

The paper in organised as follows. Section 1 is devoted to the derived category of 
equivariant sheaves on the double flag variety. We recall the definition of standard objects attached to braid group elements and how they behave under convolution. We treat the particular case of the object attached to the full twist $\bpi$. In Section 2 we define the induction functor $\Ind_F$ as well as its variant $\Indu_F$. We state the properties of these functors that we will need in our application. The relation between the cohomology of Deligne--Lusztig varieties is made in Section 3. The last section is devoted to proving our three main results.

\bigskip

\noindent \textbf{General notation.} 
We fix a prime number $p$ and denote by $\bbF_{p}$ the finite field with $p$ elements. We fix an algebraic closure $\overline{\bbF}_p$ of $\bbF_p$. We also fix a prime number $l$ different from $p$, and we denote by $\bbK$ a field which is an algebraic extension of the field $\bbQ_l$ of $l$-adic numbers. All varieties will be defined over $\overline{\bbF}_p$ and all sheaves will be constructible $\bbK$-sheaves. Given $\bfX$ a variety,
we denote by $\bbK_\bfX$ the constant sheaf on $\bfX$ with value $\bbK$. We denote by $R\Gamma_c(\bfX)$ the complex of $l$-adic cohomology with compact support with coefficients in $\bbK$, and by $H_c^i(\bfX)$ its $i$-th cohomology group. 

\bigskip

%%%%%%%%%%%%%%%%%%%%%%%%%%%%%%%%%%
\section{The equivariant derived category $\scrD$}\label{sec:hecke}
%%%%%%%%%%%%%%%%%%%%%%%%%%%%%%%%%%

%%%%%%%%%%%%%%%%%%%%%%%%%%%%%%%%%%
\subsection{Sheaves on the flag variety}\label{ssec:flag}

We refer to \cite[\S 3]{Bez-Yun} for basic results recalled in this section.
%%%%%%%%%%%%%%%%%%%%%%%%%%%%%%%%%%
Let $\bfG$ be a connected reductive group over $\overline{\bbF}_p$ and $\flag$ be its
flag variety, a smooth projective variety. We denote by $N$ its dimension.

Let
$W$ be the Weyl group of $\bfG$, defined as the parameter set of orbits of 
$\bfG$ in its diagonal action on $\flag \times \flag$. The orbit $\bfO(w)$ associated
with $w\in W$ has dimension $N+\ell(w)$. The orbit $\bfO(1)$ is the diagonal and
$S=\{s\in W\ |\ \dim\bfO(s)=N+1\}$ is the set of simple reflections of $W$.
 Given $w,w',w''\in W$ we have $ww'=w''$
and $\ell(w)+\ell(w')=\ell(w'')$ if and only if the map
$(B_1,B_2,B_3)\mapsto (B_1,B_3)$, defines an isomorphism
$$\tau_{w,w'}:\bfO(w)\times_{\flag}\bfO(w')\simto \bfO(w'').$$

Note that the orbit $\bfO(w_0)$ corresponding to the longest element 
in $W$ is dense in $\flag \times \flag$ and $N = \ell(w_0)$. 

\medskip

We denote by $\scrD$ the bounded $\bfG$-equivariant derived category 
$\scrD = D_\bfG^b(\flag \times \flag)$. Given $w\in W$, let
$j_w : \bfO(w) \hookrightarrow \flag\times\flag$ be the inclusion map.
Consider the following objects in $\scrD$
  $$ \begin{aligned}
	\Delta(w) &\, =  (j_w)_! \bbK_{\bfO(w)}[\ell(w)+N], \\
	\nabla(w) &\, =  (j_w)_* \bbK_{\bfO(w)}[\ell(w)+N],\\
	L(w) &\, =  (j_w)_{!*} \bbK_{\bfO(w)}[\ell(w)+N].
  \end{aligned}$$
Since $j_w$ is an affine embedding they are $\bfG$-equivariant perverse sheaves.

\smallskip

The category $\scrD$ is endowed with a convolution. Given $\calF,\calG \in \scrD$, the convolution of $\calF$ with $\calG$ is defined by 
  $$ \calF \odot \calG = (p_{13})_!(p_{12}^* \calF \otimes p_{23}^* \calG)[-N],$$
where $p_{ij} : \flag\times \flag \times \flag \longrightarrow \flag \times \flag$, $(B_1,B_2,B_3) \mapsto (B_i,B_j)$. The convolution $\odot$ endows $\scrD$ with a structure of monoidal category with unit object $\D(1)$.

\smallskip

Given $w$ and $w'$ two elements of $W$ such that $\ell(ww') = \ell(w) + \ell(w')$,
the isomorphism $\tau_{w,w'}$ induces canonical isomorphisms
  \begin{equation}\label{eq:mult}
	\Delta(w) \odot \Delta(w') \longiso \Delta(ww') \quad \text{and} 
	\quad \nabla(w) \odot \nabla(w') \longiso \nabla(ww').
  \end{equation} 
In addition, we have canonical isomorphisms (cf e.g. \cite[\S 11.1]{Rou})
  \begin{equation}\label{eq:inv}
	\Delta(w) \odot \nabla(w^{-1}) \longiso \Delta(1) \isolong \nabla(w^{-1}) \odot \Delta(w).
  \end{equation}

%%%%%%%%%%%%%%%%%%%%%%%%%%%%%%%%%%
\subsection{Standard objects attached to braids}\label{ssec:standard} 
%%%%%%%%%%%%%%%%%%%%%%%%%%%%%%%%%%

We fix a set $\bfW$ in bijection with $W$ via $\bfw \mapsto w$.

The Artin--Tits braid monoid $B_W^+$ is the quotient of the free monoid on $\bfW$ by the relations $\bfw = \bfw_1 \bfw_2$ whenever $w = w_1 w_2$ in $W$ and $\ell(w) = \ell(w_1) + \ell(w_2)$. 

The braid group $B_W$ is defined as the group of fractions of $B_W^+$. The length on $W$ extends to a morphism of groups $\ell : B_W \longrightarrow \bbZ$ such that $\ell(\bfw) = \ell(w)$ for every $w\in W$.

Let $\mathbf{S}$ be the set of elements of $\bfW$ corresponding to $S$ under the bijection $\bfW \simto W$. The braid monoid $B_W^+$ has the classical presentation by braid generators and braid relations:
  $$B_W^+ = \big\langle \mathbf{s} \in \mathbf{S} \, | \, \underbrace{\mathbf{sts} \cdots}_{m_{st}} = \underbrace{\mathbf{tst} \cdots}_{m_{st}} \big\rangle$$
where $m_{st}$ is the order of $st$ in $W$ for $s,t \in S$.

\smallskip

As in \cite[2.1.2]{DMR06} we consider an enlarged version $\underline{B}_W^+$ of $B^+_W$ obtained by adding generators $\underline{\bfw}$ for every $\bfw \in \bfW$. These are subject to the following relations, for every $\bfw,\bfw' \in \bfW$:
  \begin{itemize}
	 \item $\underline{\bfw} \, \underline{\bfw'} = \underline{\mathbf{ww'}}$ if
		 no element of $\mathbf{S}$ occurs both in a decomposition of
		  $\bfw$ and $\bfw'$ (note that in this case $\mathbf{ww'} \in \bfW$),
	 \item $\bfw \underline{\bfw'} = \underline{\bfw'} \bfw$ if $wv = vw$ in $W$ and $\ell(wv) = \ell(w) + \ell(v)$ for all $v \le w'$ 
 (in particular, $\mathbf{ww}' = \bfw'\bfw$ as elements in $\bfW$).  
  \end{itemize}
The length function on $B_W^+$ extends to a morphism of monoids 
$\ell:\underline{B}_W^+\to\bbZ_{\ge 0}$ given by $\ell(\underline{\bfw})=l(w)$ for
$w\in W$.

\smallskip

Given $\bfw \in \bfW$ we set $\bfO(\bfw) = \bfO(w)$ and $\bfO(\underline{\bfw}) = \overline{\bfO(w)}$. More generally, given $\bfx = (\bfx_1,\ldots,\bfx_r) \in (\bfW \cup \underline{\bfW})^r$ we set
  $$\bfO(\bfx) = \bfO(\bfx_1) \times_{\flag} \bfO(\bfx_2) \times_{\flag} \cdots \times_{\flag} \bfO(\bfx_r).$$
Following \cite{De, BM,DMR06} we can associate a variety $\bfO(\bfb)$ to every element $\bfb \in \underline{B}_W^+$. It is defined as the projective limit of the varieties $\bfO(\bfx)$ where $\bfx$ runs over the sequences of elements of $\bfW \cup \underline{\bfW}$ with product $\bfb = \bfx_1 \cdots \bfx_r$. In addition, the first and last projection on $\flag$ yield a morphism $j_\bfb : \bfO(\bfb) \longrightarrow \flag \times \flag$ from which one can define the objects $\Delta(\bfb)$ and $\nabla(\bfb)$ as in \S\ref{ssec:flag}. By construction we get canonical isomorphisms
  \begin{equation}\label{eq:multbis}
	\Delta(\bfb) \odot \Delta(\bfb') \longiso \Delta(\bfb\bfb') \quad \text{and} \quad \nabla(\bfb) \odot \nabla(\bfb') \longiso \nabla(\mathbf{bb}')
  \end{equation}
for every $\bfb,\bfb' \in \underline{B}_W^+$ which generalize the isomorphisms \eqref{eq:mult}. 

\smallskip

Let $\bfb \longmapsto \bfb^*$ be the anti-involution on $B_W^+$ that is the identity
on $\mathbf{S}$. It lifts 
the anti-involution $w \longmapsto w^{-1}$ on $W$. More concretely, if $\bfb= \mathbf{s}_1 \cdots \mathbf{s}_r$ is a decomposition into elements in $\mathbf{S}$ then $\bfb^* = \mathbf{s}_r \cdots \mathbf{s}_1$. The isomorphisms \eqref{eq:inv} extend to isomorphisms
  \begin{equation}\label{eq:invbis}
	\Delta(\bfb) \odot \nabla(\bfb^*) \longiso \Delta(1) \isolong \nabla(\bfb^*) \odot \Delta(\bfb)
  \end{equation}
for every $\bfb \in B_W^+$. This allows to define an object $\Delta(\bfb)$ in $\scrD$ associated to any element $\bfb$ in the braid group $B_W$ satisfying $\Delta(\bfb^{-1}) \simeq \nabla(\bfb^*)$ whenever $\bfb \in B_W^+$. We also have canonical isomorphisms $\D(\bfb) \odot \D(\bfc) \simto \D(\bfb\bfc)$ for all $\bfb$, $\bfc \in B_W$. Note that the isomorphisms \eqref{eq:invbis} do not hold for elements in the enriched monoid $\underline{B}^+_W$ in general. 

\smallskip

Given $w \in W$, the convolution on the left by $\Delta(w)$ and $\nabla(w^{-1})$ are 
inverse to each other by~(\ref{eq:inv}). We deduce that for any 
$\calF,\calG \in \scrD$ and $\bfb \in B_W$ we have 
  \begin{equation}\label{eq:duality}
	 \Hom_\scrD(\Delta(\bfb) \odot \calF,\calG) \simeq \Hom_\scrD(\calF, \Delta(\bfb^{-1}) \odot \calG).
  \end{equation} 
Note that the same holds for the convolution on the right.  

%%%%%%%%%%%%%%%%%%%%%%%%%%%%%%%%%%
\subsection{Two-sided cells}
%%%%%%%%%%%%%%%%%%%%%%%%%%%%%%%%%%
We denote by $\preccurlyeq$ the two-sided preorder on $W$ defined by Kazhdan--Lusztig~\cite{KL} (and denoted by $\mathop{\leqq}\limits_{^{LR}}$ there). We write $w \sim w'$ if $w \preccurlyeq w'$ and $w' \preccurlyeq w$. The two-sided cells of $W$ are by definition the equivalence classes for this relation. The preorder induces an order on two-sided cells which we will still denote by $\preccurlyeq$. Recall that $\{1\}$ and $\{w_0\}$ are two-sided cells and that $\{1\}$ is the maximal one while $\{w_0\}$ is the minimal one for $\preccurlyeq$. We will also write $w \prec w'$ if $w \preccurlyeq w'$ and $w \not\sim w'$. Following \cite[\S2]{lusztig affine} one can attach a numerical invariant $a_\Gamma \in \bbZ_{\geq 0}$ to any two-sided cell $\Gamma$ of $W$. It satisfies $a_{\Gamma'} \le a_{\Gamma}$ whenever $\Gamma \preccurlyeq \Gamma'$ (see \cite[Th. 5.4]{lusztig affine}). We have $a_{\{1\}} = 0$ and $a_{\{w_0\}} = N$.

%%%%%%%%%%%%%%%%%%%%%%%%%%%%%%%%%%
\subsection{Filtration, action of $\Delta(\bpi)$} 
%%%%%%%%%%%%%%%%%%%%%%%%%%%%%%%%%%
Given a two-sided cell $\Gamma$ of $W$ we can form the following thick subcategories $\scrD_{\preccurlyeq \Gamma}$ and $\scrD_{\prec \Gamma}$ of $\scrD = D_\bfG^b(\flag \times \flag)$: an object $\calF$ of $\scrD$ belongs to $\scrD_{\preccurlyeq \Gamma}$ (resp. $\scrD_{\prec \Gamma}$) if all the composition factors of all the perverse sheaves $\lexp{p}{H}^i(\calF)$ are of the form $L(w)$ for some $w \preccurlyeq \Gamma$ (resp. $w \prec \Gamma$). These categories are stable under convolution by any element of $\scrD$ on the left or on the right \cite[Lemma 1.4 (b)]{lusztig truncated}.
In other words, the categories $\scrD_{\preccurlyeq \Gamma}$, $\scrD_{\prec \Gamma}$ and $\scrD_{\preccurlyeq \Gamma}/\scrD_{\prec \Gamma}$ have a structure of bimodule category over $\scrD$ for the convolution.

\smallskip

Given $\bfb \in B_W$, the convolution by $\Delta(\bfb)$ induces a self-equivalence which respects the filtration by two-sided cells. Consequently, it induces a self-equivalence of monoidal categories
  $$ \begin{array}{rcl} 
	\scrD_{\preccurlyeq \Gamma}/\scrD_{\prec \Gamma} & \mathop{\longrightarrow}\limits^\sim & \scrD_{\preccurlyeq \Gamma}/\scrD_{\prec \Gamma}\\[4pt]
	\calF & \longmapsto & \Delta(\bfb) \odot \calF. 
  \end{array}$$
Let $\mathbf{w_0}$ be the lift in $B_W^+$ of the longest element in $W$ and let
$\bpi = (\mathbf{w_0})^2$. The element $\bpi$ is a central element of $B_W$ called the \emph{full twist}. The following theorem \cite[Remark 4.3]{charDmod}
describes the action of $\Delta(\bpi)$.

\begin{theorem}[Bezrukavnikov--Finkelberg--Ostrik]\label{thm:actionofpi}
	The functor induced by $ \Delta(\bpi) \odot -$ on
	$\scrD_{\preccurlyeq \Gamma}/\scrD_{\prec \Gamma}$ is isomorphic to the shift
	functor $[-2a_{\Gamma}]$.
\end{theorem}

\begin{example}
(i) Let $\Gamma =\{1\}$ be the highest two-sided cell. The thick subcategory
	 $\scrD_{\prec \Gamma}$ contains $L(w)$ for $w\neq 1$. Given $s \in S$, the perverse sheaf $\Delta(s)$ is an extension of $L(s)$ by $L(1) = \Delta(1)$. Therefore $\Delta(s) \simeq L(1)$ in $\scrD/\scrD_{\prec \Gamma}$ and consequently
  $$\Delta(\bpi) \odot L(1) \simeq \Delta(\bpi) \simeq L(1) \quad \text{in } \ \scrD/\scrD_{\prec \Gamma}.$$
 
\noindent (ii) Let $\Gamma =\{w_0\}$. The perverse sheaf $L(w_0)$ is the constant sheaf on $\flag \times \flag$, shifted by $2N$. Let $j : \bfO(w_0)\times \flag \hookrightarrow \flag^3$ be the inclusion map. We have 
  $$p_{12}^* \Delta(w_0) \otimes p_{23}^* L(w_0) \simeq j_! \bbK_{\bfO(w_0)\times \flag}[4N].$$
The restriction of $p_{13}$ to $\bfO(w_0)\times \flag$ has fibers isomorphic to an affine space of dimension $N$, therefore $(p_{13})_! j_! \bbK_{\bfO(w_0)\times \flag} \simeq \bbK_{\flag\times \flag}[-2N]$ and we deduce that
  $$\Delta(w_0) \odot L(w_0) \simeq \bbK_{\flag\times \flag}[4N-2N-N] = L(w_0)[-N].$$
We recover the fact that $\Delta(\bpi) \odot L(w_0) \simeq L(w_0)[-2N]$. 
\end{example}

We will use this theorem for objects $\calF \in  \scrD_{\preccurlyeq \Gamma}$ which lie in the (right) orthogonal of $\scrD_{\prec \Gamma}$.
%In that case the isomorphism between $\Delta(\bpi) \odot \calF$ and $\calF[-2a_{\Gamma}]$ lifts in $\scrD_{\preccurlyeq \Gamma}$.

\begin{corollary}\label{cor:actionoftpi}
	Let $\calF \in \scrD_{\preccurlyeq \Gamma}$. Assume that $\Hom_\scrD(\calG, \calF) = 0$ for every $\calG \in \scrD_{\prec \Gamma}$. We have 
	  $$ \Delta(\bpi) \odot \calF \simeq \calF[-2a_{\Gamma}]$$
	in $\scrD_{\preccurlyeq \Gamma}$.
\end{corollary}

\begin{proof}
With the assumptions on $\calF$, the quotient functor $\scrD_{\preccurlyeq \Gamma} \rightarrow \scrD_{\preccurlyeq \Gamma}/\scrD_{\prec \Gamma}$ induces an isomorphism of functors $\Hom_{\scrD_{\preccurlyeq \Gamma}}(-,\calF) \simto \Hom_{\scrD_{\preccurlyeq \Gamma}/\scrD_{\prec \Gamma}}(-,\calF)$ (see for example \cite[Ex. 10.15(ii)]{KS}). Consequently the corollary follows from Theorem \ref{thm:actionofpi}.
\end{proof}

%%%%%%%%%%%%%%%%%%%%%%%%%%%%%%%%%%
\section{Sheaves on $\bfG F$}
%%%%%%%%%%%%%%%%%%%%%%%%%%%%%%%%%%
Throughout this section we fix an endomorphism $F$ of $\bfG$ such that some power $F^\delta$ of $F$ is a Frobenius endomorphism defining an $\bbF_q$-structure on $\bfG$.
%We will call such an endomorphism $F$ a \emph{Frobenius root}. 
The finite group $\bfG^F$ is a finite reductive group. Note that $F$ acts on $\flag$, on $W$ and on $B_W$.

%%%%%%%%%%%%%%%%%%%%%%%%%%%%%%%%%%
\subsection{Induction and restriction}
%%%%%%%%%%%%%%%%%%%%%%%%%%%%%%%%%%
The group $\bfG$ acts by conjugation on the coset $\bfG F$ (in $\bfG \rtimes \langle F \rangle$) by $h (gF)h^{-1} = hgF(h^{-1}) F$. Let $D_\bfG^b(\bfG F)$ be the bounded equivariant derived category of sheaves on the coset $\bfG F$ for this action. By the Lang--Steinberg theorem, the action of $\bfG$ on $\bfG F$ is transitive. Moreover, the stabilizer of $F$ for this action is the finite group $\bfG^F$. Therefore the functor
``fiber at $F$''
  \begin{equation}\label{eq:unipsheaves}
	\begin{array}{rcl} 
		D_\bfG^b(\bfG F) & \mathop{\longrightarrow}\limits^\sim & D^b(\bbK\bfG^F\text{-mod}) \\ 
		\calF & \longmapsto & \calF_F 
	\end{array}
  \end{equation}
gives an equivalence between $D_\bfG^b(\bfG F)$ and the bounded derived category of finite-dimensional representations of $\bfG^F$ over $\bbK$. In particular the category $D_\bfG^b(\bfG F)$ is semi-simple since $\bbK$ has characteristic zero.

\smallskip

Following \cite{Lu15} we define induction and restriction functors between $\scrD$ and $D_\bfG^b(\bfG F)$. Let us consider the following diagram
  $$\xymatrix{ & \ar[ld]_f \flag\times \bfG F \ar[rd]^\varpi& \\ 
  \flag \times \flag & & \bfG F }$$
where $f(B,gF) = (B,{}^{g} F(B))$ and $\varpi(B,gF) = gF$. Since $f$ and $\varpi$ are $\bfG$-equivariant morphisms they induce functors between the equivariant derived categories. If we set $\Ind_F = \varpi_! f^*$ and $\Res_F = f_* \varpi^!$ then we get an adjoint  pair of functors
  $$ \xymatrix{D^b_\bfG(\bfG F) \ar@/_1.0pc/[rr]_{\Res_F} \ar@/^1.0pc/@{<-}[rr]^{\Ind_F} & &  \scrD= D^b_\bfG(\flag\times \flag) .}$$
We denote by $\scrU$ the thick subcategory of $D_\bfG^b(\bfG F)$ generated by the image of $\Ind_F$. Under the equivalence $D_\bfG^b(\bfG F) \simeq D^b(\bbK\bfG^F\text{-mod})$, it corresponds to the bounded derived category of unipotent representations of $\bfG^F$ over $\bbK$. 

%%%%%%%%%%%%%%%%%%%%%%%%%%%%%%%%%%
\subsection{Filtration}\label{sec:filtration}
%%%%%%%%%%%%%%%%%%%%%%%%%%%%%%%%%%
>From now on we assume that $\bbK$ is large enough for $\bfG^F$, so that $\bbK\bfG^F$ is split semisimple. To a unipotent character $\rho$ of $\bfG^F$, Lusztig associates two invariants $a_\rho$ and $A_\rho$ (roughly speaking, they are the valuation and the degree 
of the polynomial in $q^{1/\delta}$ representing the dimension of $\rho$). The unipotent characters of $\bfG^F$ fall into families, which are in turn labeled by 
$F$-stable two-sided cells of $W$. The two-sided cell corresponding to a unipotent character $\rho$ will be denoted by $\Gamma_{\rho}$. 
Given a two-sided cell $\Gamma$ of $W$ we define $\scrU_{\preccurlyeq \Gamma}$  (resp. $\scrU_{\prec \Gamma}$) to be the thick subcategory of $D_\bfG^b(\bfG F)$
generated by the image under the inverse of \eqref{eq:unipsheaves} of the unipotent
$\bbK\bfG^F$-modules
$\rho$ such that $\Gamma_\rho \preccurlyeq \Gamma$ (resp. $\Gamma_\rho \prec \Gamma$). 

\smallskip

The following proposition is proved in \cite[Prop. 2.4(a)]{Lu15}. Note that our definition of $\Res_F$ differs from the one of Lusztig who uses $f_! \varpi^*$ instead. However, the two definitions are exchanged by Poincar\'e--Verdier duality, which preserves the filtrations on both $\scrD$ and $\scrU$.

\begin{proposition}\label{prop:compatibility}
	Let $\Gamma$ be an $F$-stable two-sided cell of $W$. The functors $\Ind_F$ and $\Res_F$ restrict to functors 
	  $$ \xymatrix{\scrU_{\preccurlyeq \Gamma} \ar@/_/[rr]_{\Res_F} \ar@/^/@{<-}[rr]^{\Ind_F} & &  \scrD_{\preccurlyeq \Gamma}} \quad 
	  \text{and} \quad \xymatrix{\scrU_{\prec \Gamma} \ar@/_/[rr]_{\Res_F} \ar@/^/@{<-}[rr]^{\Ind_F} & &  \scrD_{\prec \Gamma} .}$$
\end{proposition}

\begin{rmk}\label{rmk:avalue}
The map $w \longmapsto w w_0$ induces an order reversing bijection on families. As explained in \cite[\S 1.3]{Lu15} the map $\rho \longmapsto \Gamma_\rho$ differs from the original definition of Lusztig (see \cite[\S4]{LuBook}) by multiplication by $w_0$. In particular we have $a_{\rho} = a_{\Gamma_\rho w_0}$.

\smallskip 

As an example, let us consider the perverse sheaf $L(w_0)$ which is, up to a shift, the constant sheaf on $\flag \times \flag$. Its image by $\Ind_F$ has only
the constant sheaf  $\bbK_{\bfG F}$ as a composition factor. Under the equivalence $D_\bfG^b(\bfG F) \simeq D^b(\bbK\bfG^F\text{-mod})$, this constant sheaf
corresponds to the trivial unipotent character $1_{\bfG^F}$. With our previous notation this means that $\Gamma_{1_{\bfG^F}} = \{w_0\}$ is the lowest two-sided cell. Note that $a_{1_{\bfG^F}} = 0$ whereas $a_{\{w_0\}} = N$. 
\end{rmk}

%%%%%%%%%%%%%%%%%%%%%%%%%%%%%%%%%%
\subsection{Further properties of the functor $\Ind_F$}\label{ssec:trace}
%%%%%%%%%%%%%%%%%%%%%%%%%%%%%%%%%%
Given a variety $\bfX$ acted on by $\bfG$, we will denote by 
  $$\For_{\bfG^F}^\bfG : D_\bfG^b(\bfX) \longrightarrow D_{\bfG^F}^b(\bfX)$$
the forgetful functor. Let $\bfY$ be another variety which is only acted on by the finite group $\bfG^F$. We consider the induced variety $\bfX = \bfG \times_{\bfG^F} \bfY$ and the inclusion $\iota_\bfY : \bfY \hookrightarrow \bfX$. The action of $\bfG$ on $\bfX$ is given by left multiplication on the left factor and the map $\iota_\bfY$ is $\bfG^F$-equivariant. By \cite[\S 2.6.3]{BL} there is an equivalence of categories 
  $$  \iota_{\bfY}^* \circ \For_{\bfG^F}^\bfG  :   D_\bfG^b(\bfG \times_{\bfG^F} \bfY)  \longiso D_{\bfG^F}^b(\bfY).$$
When $\bfY = \{F\}$ we recover the equivalence \eqref{eq:unipsheaves}. 

\begin{proposition}\label{prop:ind-forget}
	Let $\iota_F :  \{F\} \hookrightarrow \bfG F$ and $\iota : \flag \hookrightarrow \flag \times \flag$ be the inclusion of the graph of $F$ in the double flag variety. There is an isomorphism of functors
    $$  \iota_{F}^* \circ \For_{\bfG^F}^\bfG \circ \Ind_F \, \mathop{\Longrightarrow}\limits^\sim \,  R\Gamma (\flag, \iota^* \circ \For_{\bfG^F}^\bfG(-)).$$
\end{proposition}

\begin{proof}
	Let $\varpi':\flag\to\{F\}$ be the structure map and define
	$\iota':\flag\to \flag \times \bfG F$ by $\iota'(B)=(B,F)$.
There is a proper base change isomorphism
	$\iota_F^* \varpi_!  \mathop{\Longrightarrow}\limits^\sim {\varpi'}_! (\iota')^*$ and a canonical
	isomorphism $(\iota')^* f^*  \mathop{\Longrightarrow}\limits^\sim 
	(f \iota')^* = \iota^*$ making
the following diagram commutative. Note also that the right square is cartesian:
  $$\xymatrix{ & & \ar@{_(->}[dl]_{\iota'} \ar[dr]^{\varpi'} \ar@/_2.0pc/[ddll]_{\iota} \flag \\ 
  & \flag \times \bfG F \ar[dr]^\varpi \ar[dl]_f & & \{F\} \ar@{_(->}[dl]_{\iota_F} \\ 
  \flag \times\flag & & \bfG F}$$
Note that the forgetful functor commutes with the pull-back and push-forward functors.
So we have a canonical isomorphism $\For_{\bfG^F}^\bfG \circ \Ind_F 
\mathop{\Longrightarrow}\limits^\sim \varpi_!\circ f^*\circ \For_{\bfG^F}^\bfG$.
The proposition follows by composing those three isomorphisms.
\end{proof}

Proposition \ref{prop:ind-forget} shows that the induction functor $\Ind_F$ factors through the forgetful functor. From now on we will
work with $\bfG^F$-equivariant sheaves and the induction functor 
  $$\Indu_F = R\Gamma(\flag,  \iota^*(-)) : D_{\bfG^F}^b(\flag \times \flag) \longrightarrow D^b(\bbK\bfG^F\text{-mod}).$$
With this notation, Proposition \ref{prop:ind-forget} can be rephrased into the existence of a natural isomorphism
  \begin{equation}\label{eq:ind-stalk}
	\big(\Ind_F(\calF)\big)_F \longiso \Indu_F(\For_{\bfG^F}^\bfG \calF)
  \end{equation}
for all $\calF \in \scrD$. Unless there is a risk of confusion, we will continue to denote by $\Delta(\bfb)$ the image under the forgetful functor of the standard objects attached to elements of the braid group. In particular \eqref{eq:ind-stalk} will be written
  $$(\Ind_F  \Delta(\bfb))_F \longiso \Indu_F  \Delta(\bfb).$$
The convolution $\odot$ defined in \S\ref{ssec:flag} makes sense also in the  $\bfG^F$-equivariant derived category, and again we shall use the same notation. Finally, since
$F$ is $\bfG^F$-equivariant, it induces canonical isomorphisms 
  \begin{equation}\label{eq:action of F}
	F^*(\D(\bfb)) \longiso \D(F^{-1}(\bfb))\quad\text{and}\quad F_*(\D(\bfb)) \longiso \D(F(\bfb))
  \end{equation}
in the category $D_{\bfG^F}^b(\flag\times\flag)$. Note that it is crucial to work with $\bfG^F$-equivariant sheaves here, as such isomorphisms do not make sense in
the category $\scrD$.

\smallskip

Brou\'e and Michel~\cite[\S{2.A}]{BM} constructed $\bfG^F$-equivariant morphisms between Deligne--Lusztig varieties inducing isomorphisms between their cohomology groups. In our setting this can be rephrased as an isomorphism 
  \begin{equation}\label{eq:bm}
	\Indu_F\big(\Delta(\bfc) \odot \Delta(F(\bfb))\big) \longiso \Indu_F\big(\Delta(\bfb) \odot \Delta(\bfc)\big)
  \end{equation}
in $D^b(\bbK\bfG^F\text{-mod})$, for any $\bfb,\bfc \in B_W^+$. Following a result of Lusztig in the case of character sheaves~\cite[1.11(a)]{lusztig truncated} we construct such an isomorphism for general objects in the equivariant derived category $D_{\bfG^F}^b(\flag \times \flag)$. We expect our construction to generalise the one of Brou\'e-Michel
but we did not check that.

\begin{proposition}\label{prop:trace}
	Let $\calF,\calG \in D_{\bfG^F}(\flag\times\flag)$. Then there exists an isomorphism
	  $$ c_{\calF,\calG} : \Indu_F(\calF \odot \calG) \longiso \Indu_F(F^*(\calG) \odot \calF)$$
	which is natural in $\calF$ and $\calG$.
\end{proposition}

\begin{proof}
Let us first consider the following commutative diagram, where the bottom right square is cartesian
  \begin{equation}
	\xymatrix{ && \flag \times \flag \ar[lld]_{q_{ij}} \ar[ld]^{\iota'} \ar[d]^{p_{1}} \ar@(r,u)[]_\theta&  \\ 
	\flag \times \flag & \flag \times \flag \times \flag \ar[l]^{p_{ij}} \ar[d]^{p_{13}} & \flag \ar[ld]^{\iota} \\ 
	& \flag \times \flag } 
  \end{equation}
and the maps are given by
  $$\begin{cases}
	\iota'(B_1,B_2)=(B_1, B_2,F(B_1)),\\
	p_{ij}(B_1,B_2,B_3)=(B_i,B_j),\\
	p_{1}(B_1,B_2)=B_1,\\
	q_{ij}=p_{ij} \circ \iota',\\
	\theta(B_1,B_2) = (B_2,F(B_1)).
  \end{cases}$$
Given $\calF$ and $\calG$ two objects in $D_{\bfG^F}(\flag\times\flag)$, we have natural isomorphisms
  \begin{align*}
	&& \iota^*((p_{13})_!(p_{12}^* \calF \otimes p_{23}^* \calG)) 
	  \longiso\ & (p_{1})_! \iota^{\prime *}(p_{12}^* \calF \otimes p_{23}^* \calG)  && \text{(base change)}\\
	  &&  \longiso\ & (p_{1})_! (q_{12}^* \calF \otimes q_{23}^* \calG) && \text{(composition)} \\
	  &&  =\ & (p_{1})_! (\calF \otimes \theta^* \calG) &&
  \end{align*}
since $q_{12} = \mathrm{Id}_{\flag \times \flag}$ and $q_{23} = \theta$. We denote by $\eta$ the isomorphism of functors obtained after taking global sections:
  $$ \eta : \Indu_F (-_1 \odot -_2 ) \mathop{\Longrightarrow}\limits^\sim R\Gamma(\flag \times \flag, -_1 \otimes \theta^*(-_2))[-N].$$
As $\theta$ is an equivalence of \'etale sites, the unit ${\boldsymbol{1}} \Longrightarrow \theta_*\theta^*$ is a natural isomorphism of functors (and note that $\theta_*=\theta_!$). Consequently we obtain another sequence of natural isomorphisms
  \begin{align*}
	&&\calF \otimes \theta^* \calG   \longiso\ & \theta_! \theta^* (\calF \otimes \theta^* \calG) && \text{(unit)}\\
	&& \longiso\ &  \theta_!(\theta^* \calF \otimes F^* \calG) && \text{(composition)}\\
	&& \longiso\ & \theta_!( F^* \calG \otimes \theta^* \calF). && \text{(symmetry)}
  \end{align*}
Note that we have used that $\theta^2 = F$. Combining these with the isomorphism $R\Gamma(\flag,\theta_!(-)) \mathop{\Longrightarrow}\limits^\sim  R\Gamma(\flag,-)$ coming from the composition of push-forwards, we obtain a natural transformation $\gamma$
given by
  $$\gamma : R\Gamma(\flag \times \flag, -_1 \otimes \theta^*(-_2)) \longiso R\Gamma(\flag \times \flag, F^*(-_2) \otimes \theta^*(-_1)).$$
The composition $c = \eta^{-1} \circ \gamma \circ \eta$ gives the required natural
isomorphism. \end{proof}

%%%%%%%%%%%%%%%%%%%%%%%%%%%%%%%%%%
\section{Deligne--Lusztig varieties and their cohomology}
%%%%%%%%%%%%%%%%%%%%%%%%%%%%%%%%%%
Let $w \in W$. Recall from \S\ref{ssec:flag} that $\bfO(w)$ denotes the $\bfG$-orbit on $\flag \times\flag$ associated to $w$. The intersection of $\bfO(w)$ with the graph $\boldsymbol \Gamma_F$ of $F$ is the Deligne--Lusztig variety $\bfX(wF)$ associated to $wF$
  $$\bfX(wF) = \bfO(w) \cap \boldsymbol\Gamma_F = \{ (B,B') \in \bfO(w) \, | \, B' = F(B)\}.$$
It is a smooth quasi-projective variety of dimension $\ell(w)$. We denote by $H^i(\bfX(wF))$ (resp. $H_c^i(\bfX(wF))$) the $i$-th cohomology group (resp. cohomology group with compact support) of $\bfX(wF)$ with coefficients in $\bbK$. 
We denote by $IH^i(\overline{\bfX(wF)})$   the $i$-th  intersection cohomology group
of $\overline{\bfX(wF)}$ with coefficients in $\bbK$.
The action of $\bfG$ on $\bfO(w)$ restricts to an action of $\bfG^F$ on $\bfX(wF)$. Consequently, the three types of cohomology groups above are representations of $\bfG^F$ over $\bbK$.

\smallskip

Recall from \S\ref{ssec:standard} that one can also attach a variety $\bfO(\bfb)$ to any element $\bfb$ of the braid monoid ${B}_W^+$. We define more generally the variety $\bfX(\bfb F)$ by the following cartesian square:
  \begin{equation}\label{eq:dlvariety}
	\xymatrix{ \bfX(\bfb F) \ar@{^(->}[r] \ar[d]^{j_\bfb'} & \bfO(\bfb) \ar[d]^{j_\bfb} \\ 
	\boldsymbol\Gamma_F \ar@{^(->}[r]^-{\iota} & \flag\times\flag }
  \end{equation}
where $\iota : \boldsymbol \Gamma_F \hookrightarrow \flag\times\flag$ is the inclusion of the graph of $F$. The various cohomology groups of Deligne--Lusztig varieties can be obtained through $\Ind_F$ as the image of the objects in $\scrD$ defined in \S\ref{ssec:flag} and \S\ref{ssec:standard}. For convenience we will state the result for the  functor $\Indu_F$ defined in \S\ref{ssec:trace} instead of $\Ind_F$ (see also \eqref{eq:ind-stalk}). 

\begin{proposition}\label{prop:imageofind}
	We have, for all $\bfb \in {B}_W^+$ and $w \in W$, 
	\begin{itemize}
  	 	\item[$\mathrm{(i)}$] $\Indu_F(\Delta(\bfb)) \simeq R\Gamma_c(\bfX(\bfb F))[N+\ell(\bfb)]$, 
 	 	\item[$\mathrm{(ii)}$]  $\Indu_F(\nabla(\bfb)) \simeq R\Gamma(\bfX(\bfb F))[N+\ell(\bfb)]$, 
  		\item[$\mathrm{(iii)}$]  $\Indu_F(L(w))  \simeq IC(\overline{\bfX(w F)},\bbK_{\bfX(w F)})[N+\ell(w)]$.
	\end{itemize}
\end{proposition}

\begin{proof}
Let $\bfb \in {B}_W^+$ and $j_\bfb : \bfO(\bfb) \longrightarrow \flag\times\flag$ be the map considered in \S\ref{ssec:standard}. Since we are working with $\bfG^F$-equivariant sheaves, we shall consider $j_\bfb$ as a $\bfG^F$-equivariant map only. By base change on \eqref{eq:dlvariety}, we get an isomorphism
  $$ \iota^* (\Delta(\bfb)) = \iota^* (j_\bfb)_!(\bbK_{\bfO(\bfb)})[N+\ell(\bfb)] 
	\longiso\  (j_\bfb')_!(\bbK_{\bfX(\bfb)})[N+\ell(\bfb)] $$
from which we deduce (i) after taking the global section functor $R\Gamma(\flag,-)$. The isomorphism (ii) is obtained in a similar way using the smooth base change on \eqref{eq:dlvariety}.

\smallskip

For (iii), we follow the proof of~\cite[Lemma~4.3]{lusztig odd}. We fix an $F$-stable Borel subgroup $\bfB_0$ of $\bfG$, an $F$-stable maximal torus $\bfT_0$ of $\bfB_0$ and we identify $\flag$ with $\bfG/\bfB_0$ and $W$ with $N_\bfG(\bfT_0)/\bfT_0$. Let $\pi : \bfG \to \bfG/\bfB_0$ be the canonical map and $\calL : \bfG \to \bfG$, $g \mapsto g^{-1} F(g)$ be the Lang map. Let $\bfX_w=\calL^{-1}(\overline{\bfB_0 w \bfB_0})$. The
	varieties $\overline{\bfB_0 w \bfB_0}$ and $\bfX_w$ are stable by right multiplication by $\bfB_0$. The map $\pi : \bfG \to \bfG/\bfB_0$ is a principal bundle and, by definition, we have $\pi(\bfX_w)=\overline{\bfX(wF)}$. Let us now consider 
the following commutative diagram:
$$\diagram 
\bfX_w \dto_{{\pi_w}} \rto^{{\varphi}\qquad} & 
\bfG \times \overline{\bfB_0 w \bfB_0} \dto^{\psi} \rto^{\quad{p_2}} & \overline{\bfB_0 w \bfB_0}\\
\overline{\bfX(wF)} \rto^{{\iota_w}} & \overline{\bfO(w)},
\enddiagram$$
	where $\pi_w$ (resp. $\iota_w$) denotes the restriction of $\pi$ (resp. $\iota$), $\varphi(g)=(g,\calL(g))$, $\psi(g,h)=({}^g\bfB_0,{}^{gh}\bfB_0)$ and $p_2$ is the second projection. By abuse of notation we will still denote by $L(w)$ its restriction to the closed subvariety  $ \overline{\bfO(w)}$ of $\flag \times \flag$. We have
	$\pi_w^*(\iota_w)^* (L(w)) \simeq \varphi^*\psi^*(L(w))$. Since
$\psi$ is a smooth map (with fibers isomorphic to $\bfB_0 \times \bfB_0$), we have
  $$\psi^*(L(w)) \simeq IC(\bfG \times \overline{\bfB_0 w \bfB_0},\bbK_{\bfG \times \bfB_0 w \bfB_0})[N + \ell(w)].$$
As $\bfG$ is smooth, this shows that 
  \begin{eqnarray*}
	\psi^*(L(w)) &\simeq& \bbK_\bfG \boxtimes IC(\overline{\bfB_0 w \bfB_0},\bbK_{\bfB_0 w \bfB_0})[N + \ell(w)] \\
	& \simeq & p_2^* IC(\overline{\bfB_0 w \bfB_0},\bbK_{\bfB_0 w \bfB_0})[N + \ell(w)].
  \end{eqnarray*}
The map $p_2 \circ \varphi$ is the restriction of $\calL$ to a map $\bfX_w \longrightarrow \overline{\bfB_0 w \bfB_0}$. 
Since this map is \'etale (hence smooth), one gets that 
  $$\pi_w^*(\iota_w)^*(L(w)) \simeq IC(\bfX_w,\bbK_{\calL^{-1}(\bfB_0 w \bfB_0)})[N + \ell(w)].$$
As $\pi_w$ is a principal bundle, this forces 
  $$(\iota_w)^* L(w) \simeq IC(\overline{\bfX(w F)},\bbK_{\bfX(w F)})[N+\ell(w)],$$
as expected.
\end{proof}

%%%%%%%%%%%%%%%%%%%%%%%%%%%%%%%%%%
\section{Applications}\label{sec:applications}
%%%%%%%%%%%%%%%%%%%%%%%%%%%%%%%%%%

%%%%%%%%%%%%%%%%%%%%%%%%%%%%%%%%%%
\subsection{Invariance under conjugation by $B_W$}
%%%%%%%%%%%%%%%%%%%%%%%%%%%%%%%%%%
Let $\bfb,\bfc \in B_W$. Using the results of the previous sections we obtain
a sequence of isomorphisms in $D^b(\bbK\bfG^F\text{-mod})$
  \begin{align*}
	&&  \Indu_F \D(\bfc \bfb F(\bfc^{-1})) & \longiso \Indu_F \big(\D(\bfc) \odot \D(\bfb) \odot \D(F(\bfc^{-1}))\big) && \\
	&& & \longiso \Indu_F \big(F^* \D(F(\bfc^{-1})) \odot \D(\bfc) \odot \D(\bfb) \big) && \text{(by Prop. \ref{prop:trace})}\\
	&& & \longiso \Indu_F \big(\D(\bfc^{-1}) \odot \D(\bfc) \odot \D(\bfb) \big) && \text{(by \eqref{eq:action of F})}\\
	&& & \longiso \Indu_F \D(\bfb). &&
  \end{align*}
By Proposition \ref{prop:imageofind} this has the following consequence, which was conjectured in~\cite[Conj.~3.1.7]{DMR06}.

\begin{theorem}\label{thm:conjugacy}
	Let $\bfb,\bfb' \in B_W^+$. If $\bfb F$ and $\bfb'F$ are  conjugate under $B_W$ then for every $i \in \bbZ$ the $\bbK\bfG^F$-modules $ H_c^i\big(\bfX(\bfb F)\big)$ and  $H_c^i\big(\bfX(\bfb'F)\big)$ are isomorphic. 
\end{theorem}

%%%%%%%%%%%%%%%%%%%%%%%%%%%%%%%%%%
\subsection{Translation by the full twist $\bpi$}
%%%%%%%%%%%%%%%%%%%%%%%%%%%%%%%%%%
Recall from \S\ref{sec:filtration} that to every unipotent character $\rho$ is attached an $F$-stable two-sided cell $\Gamma_\rho$ and invariants $a_\rho$ and
$A_\rho$.
Digne--Michel--Rouquier conjectured in \cite[Conj. 3.3.24]{DMR06} a relation between the cohomology of the Deligne--Lusztig varieties $\bfX(\bfb F)$ and $\bfX(\boldsymbol\pi \bfb F)$. We give here a proof of this conjecture using Theorem \ref{thm:actionofpi}.

\begin{theorem}\label{thm:fulltwist}
	Let $\rho$ be a unipotent character of $\bfG^F$ and $\bfb \in \underline{B}_W^+$.We have 
	  $$ \big\langle \rho, H_c^i\big(\bfX(\boldsymbol\pi \bfb F)\big) \big\rangle_{\bfG^F} = \big\langle\rho, H_c^{i-4N+2A_{\rho}}\big(\bfX(\bfb F)\big)\big\rangle_{\bfG^F}$$
	for all $i \in \bbZ$.
\end{theorem}

\begin{proof}
Let $\calF_\rho \in \scrU_{\preccurlyeq \Gamma_\rho}$ be the image of $\rho$ under the
	inverse of \eqref{eq:unipsheaves}. Given $i \in \bbZ$ and $w \in W$,
Proposition \ref{prop:imageofind} shows that
  $$\begin{aligned} 
	\Hom_{\scrD} \big(L(w)[i], \Res_F(\calF_\rho)\big) &\, \simeq  \Hom_{\scrU} \big(\Ind_F(L(w))[i], \calF_\rho\big) \\
	 & \, \simeq \Hom_{\bbK\bfG^F}\big(\Indu_F(L(w))[i],\rho\big)\\
	 & \, \simeq \Hom_{\bbK\bfG^F}\big(IH^{i+N+\ell(w)}\big(\overline{\bfX(wF)}\big),\rho\big)
  \end{aligned}$$ 
which is zero unless $w \in \Gamma$ with $\Gamma_\rho \preccurlyeq \Gamma$ (see Proposition \ref{prop:compatibility}). Consequently, $\Hom_\scrD\big(\calG,\Res_F(\calF_\rho)\big) = 0$ for every $\calG \in \scrD_{\prec \Gamma_\rho}$. Using Corollary \ref{cor:actionoftpi} we get the following isomorphism in $\scrD$
  $$\Delta(\boldsymbol\pi^{-1}) \odot \Res_F(\calF_\rho) \simeq \Res_F(\calF_\rho)[2a_{\Gamma_\rho}].$$

Let $\bfb \in \underline{B}_W^+$ and $i \in \bbZ$. Using the adjunctions and the previous isomorphism we obtain
  $$\begin{aligned}
	\Hom_{\scrU}\big(\Ind_F(\Delta(\bpi \bfb))[i], \calF_\rho\big) 
	&\, \simeq \Hom_{\scrD}\big(\Delta(\bpi \bfb)[i], \Res_F(\calF_\rho)\big) \\
	&\, \simeq \Hom_{\scrD}\big(\Delta(\bfb)[i], \Delta(\bpi^{-1}) \odot\Res_F(\calF_\rho)\big) \\
	&\, \simeq \Hom_{\scrD}\big(\Delta(\bfb)[i], \Res_F(\calF_\rho)[2a_{\Gamma_\rho}]\big)\\
	&\, \simeq \Hom_{\scrU}\big(\Ind_F(\Delta(\bfb))[i-2a_{\Gamma_\rho}], \calF_\rho\big).
  \end{aligned}$$
By Proposition \ref{prop:imageofind}, under the equivalence \eqref{eq:unipsheaves}  between $D^b_\bfG(\bfG F)$ and $D^b(\bbK\bfG^F\text{-mod})$ the previous isomorphism becomes
  $$\Hom_{\bbK\bfG^F}\big(H_c^{i+N+\ell(\bpi \bfb)}\big(\bfX(\bpi \bfb F)\big),\rho\big) \simeq  \Hom_{\bbK\bfG^F}\big(H_c^{i+N+\ell(\bfb)-2a_{\Gamma_\rho}}\big(\bfX( \bfb F)\big),\rho\big)$$
and we conclude by taking the dimensions and using the equality $\ell(\bpi) +2 a_{\Gamma_\rho} = 4N-2A_{\rho}$ which comes from the definition of $A_\rho$.
\end{proof}

As a particular case, the cohomology of the Deligne--Lusztig variety associated to the full twist is given by 
  \begin{equation}\label{eq:xpi}
	  H_c^\bullet \big(\bfX(\boldsymbol\pi F)) \simeq \bigoplus_{\chi \in \mathrm{Irr}\, (W^F)} \rho_\chi^{\oplus \chi(1)}[2A_{{\rho_\chi}}-4N]
  \end{equation}
where $\rho_\chi$ is the principal series unipotent character corresponding to $\chi$. This was first conjectured by Brou\'e--Michel~\cite[Conj.~2.15(3), (4), (5)]{BM}

%%%%%%%%%%%%%%%%%%%%%%%%%%%%%%%%%%
\subsection{Cohomology for roots of the full twist in type $A$}
%%%%%%%%%%%%%%%%%%%%%%%%%%%%%%%%%%
We finish by stating some consequences of Theorem \ref{thm:fulltwist} for groups of type $A$ using the results of \cite{Du}. In this section only we assume that $\bfG = \mathrm{GL}_n$ and $F$ is the standard Frobenius endomorphism. The following result solves a conjecture of Brou\'e--Michel (see \cite[Conj. 5.7(1)]{BM}) in this case.

\begin{theorem}
	Assume that $(\bfG,F)$ is a group of type $A$. Let $d \geq 1$ and $\bfw \in B_W^+$ be a $d$-th root of $\bpi$. For all $i \neq j$, we have
	  $$ \big\langle H_c^i(\bfX(\bfw F)) , H_c^j(\bfX(\bfw F)) \big\rangle_{\bfG^F} = 0.$$
\end{theorem}
 
\begin{proof}
By \cite[Thm. 1.1]{GM} the set of $d$-th roots of $\pi$ forms a single conjugacy class. Therefore by Theorem \ref{thm:conjugacy} it is sufficient to prove it for a specific $d$-th root. By \cite[Thm. 3.12]{BM} the image of $\bfw$ in $W$ is a regular element, therefore $d$ must be a regular number (in this case a divisor of $n$ or $n-1$). Then the result follows from \cite[Cor. 3.2]{Du} and \eqref{eq:xpi}.
\end{proof} 
 
\begin{rmk}
Note that the combination of \cite[Cor. 3.2]{Du} and \eqref{eq:xpi} gives a complete proof of the conjecture stated in \cite[Conj. 1]{Du}. This implies that one can compute explicitly the cohomology of Deligne--Lusztig varieties attached to roots of $\pi$, and also their parabolic versions.
\end{rmk}


\begin{thebibliography}{99}

%\bibitem{AHR}
%{\sc P. Achar,  A. Henderson, S. Riche}, Geometric Satake, Springer correspondence, and small representations II.
%Represent. Th. {\bf 19} (2015), 94--166.

\bibitem{BL}
{\sc J. Bernstein, V. Lunts}, Equivariant sheaves and functors. Lecture Notes in Math. {\bf1578},
Springer, 1994.

\bibitem{charDmod}
	{\sc R.Bezrukavnikov, M.Finkelberg and V.Ostrik},
		Character $D$-modules via Drinfeld center of Harish-Chandra
		bimodules, Inv. Math. {\bf 188} (2012), 589--620.

\bibitem{Bez-Yun}
	{\sc R.Bezrukavnikov and Z.Yun},
		On Koszul duality for Kac-Moody groups,
		Representation Theory {\bf 17} (2013), 1--98.

\bibitem{Br}
	{\sc M.Brou\'e}, Isom\'etries parfaites, types de blocs, cat\'egories
        d\'eriv\'ees,
        Ast\'erisque {\bf 181--182} (1990), 61--92.

\bibitem{Brmit} {\sc M.~Brou\'e},
        Reflection Groups, Braid Groups, Hecke Algebras, Finite
        Reductive Groups,
        in ``Current Developments in Mathematics, 2000'', pp 1--103,
        International Press, 2001.
	
\bibitem{BrMa} {\sc M.~Brou\'e and G.~Malle},
        Zyklotomische Heckealgebren,
        Ast\'erisque {\bf 212} (1993), 119--189. 
	
\bibitem{BM}
{\sc M. Brou\'e, J. Michel}, Sur certains \'el\'ements r\'eguliers des groupes 
de Weyl et les vari\'et\'es de Deligne--Lusztig associ\'ees. 
%(French) [Some regular elements of Weyl groups and the associated Deligne--Lusztig varieties] 
Finite reductive groups (Luminy, 1994), 73--139, Progr. Math. {\bf 141}, 
Birkh\"auser Boston, Boston, MA, 1997.

\bibitem{De} 
	{\sc P.Deligne},
        Action du groupe des tresses sur une cat\'egorie,
        Inv. Math. {\bf 128} (1997), 159--175.

\bibitem{DeLu}
	{\sc P.Deligne and G.Lusztig},
        Representations of reductive groups over finite fields,
        Ann. of Math. {\bf 103} (1976), 103--161.


\bibitem{DMR06}
{\sc F. Digne, J. Michel, R. Rouquier}, Cohomologie des vari\'et\'es de Deligne--Lusztig. 
%(French) [Cohomology of Deligne--Lusztig varieties] 
Adv. Math. {\bf 209} (2007), 749--822.

\bibitem{Du}
{\sc O. Dudas}, Cohomology of Deligne--Lusztig varities for unipotent blocks of $\mathrm{GL}_n(q)$. 
%(French) [Cohomology of Deligne--Lusztig varieties] 
Represent. Th. {\bf 17} (2013), 647--662.

\bibitem{GM}
{\sc J. Gonz\'alez-Meneses}, The $n$-th root of a braid is unique up to conjugacy.
Algebr. Geom. Topol. {\bf 3} (2003), 1103--1118.

\bibitem{KS} {\sc M. Kashiwara \& P. Schapira},
\emph{Categories and sheaves}. Grundlehren der Mathematischen Wissenschaften {\bf 332}, Springer-Verlag, Berlin, 2006.

\bibitem{KL} {\sc D. Kazhdan \& G. Lusztig},
Representations of Coxeter groups and Hecke algebras, 
Inventiones {\bf 53} (1979), 165-184.

\bibitem{lusztig odd} {\sc G. Lusztig}, 
Unipotent characters of the symplectic and odd orthogonal groups over a finite field, 
Invent. Math. {\bf 64} (1981), 263--296.

\bibitem{LuBook} {\sc G. Lusztig}, 
\emph{Characters of reductive groups over a finite field.} Annals of Mathematics Studies {\bf107}, Princeton University Press, Princeton, NJ, 1984.

\bibitem{lusztig affine} {\sc G. Lusztig}, 
Cells in affine Weyl groups, in {\it Algebraic groups and related topics}, 
Adv. Stud. Pure Math. {\bf 6} (1985), 255-287.

\bibitem{lusztig truncated}
{\sc G. Lusztig},
Truncated convolution of character sheaves, 
Bull. Inst. Math. Acad. Sin. {\bf 10} (2015), 1-72. 

\bibitem{Lu15}
{\sc G. Lusztig},
Unipotent representations as a categorical centre, 
Represent. Th. {\bf 19} (2015), 211--235.

\bibitem{Rou}
	{\sc R.~Rouquier},
	Categorification of $\sl_2$ and braid groups,
	 in "Trends in representation theory of algebras and related topics",
	 137--167, Amer. Math. Soc., 2006.
\end{thebibliography}
\end{document}